\documentclass[11pt]{article}

\usepackage[T1]{fontenc}
\usepackage[utf8]{inputenc}
\usepackage{lmodern}
\usepackage{microtype}
\usepackage{amsmath,amssymb,amsthm,mathtools}
\usepackage{enumitem}
\usepackage[a4paper,margin=1.08in]{geometry}
\usepackage[hidelinks]{hyperref}
\usepackage[nameinlink,noabbrev]{cleveref}
\usepackage{booktabs}

\setlist{nosep,leftmargin=2em}

\newtheorem{theorem}{Theorem}[section]
\newtheorem{proposition}[theorem]{Proposition}
\newtheorem{corollary}[theorem]{Corollary}

\theoremstyle{definition}
\newtheorem{definition}[theorem]{Definition}
\newtheorem{problem}[theorem]{Problem}
\theoremstyle{remark}

\newcommand{\DT}{\mathcal D_T}
\newcommand{\Idl}{\operatorname{Idl}}
\newcommand{\Arith}{\operatorname{ARITH}}

\newcommand{\cl}{\operatorname{cl}}
\newcommand{\lfp}{\operatorname{lfp}}
\newcommand{\down}{\mathord{\downarrow}}
\newcommand{\join}{\mathbin{\oplus}}
\newcommand{\bigjoin}{\bigoplus}
\newcommand{\leT}{\leq_T}
\newcommand{\ltT}{<_T}
\newcommand{\equivT}{\equiv_T}

\title{\textbf{Jump Closure and Limit Uniformization\\in the Ideal Completion of the Turing Degrees}}
\author{Miara Sung\thanks{jbaek080@berkeley.edu}}
\date{August 2026}

\begin{document}
\maketitle

\begin{abstract}
The Turing jump has no fixed point on the Turing degrees: $\mathbf a <_T \mathbf a'$ for every degree $\mathbf a$. After passing to the ideal completion, however, a natural fixed-point phenomenon appears. We study the Scott-continuous lifting
$\Gamma:\operatorname{Idl}(\mathbf D_T)\to\operatorname{Idl}(\mathbf D_T)$,
given by
$\Gamma(I)=\downarrow\{\mathbf a':\mathbf a\in I\}$.
Starting from the computable degree, Kleene iteration reaches its first fixed point at stage $\omega$, namely the Turing ideal of arithmetical degrees; more generally, above $\mathbf a$ the least fixed point is the ideal of degrees arithmetical in $\mathbf a$.

To pass beyond this fixed point, we introduce a limit-uniformization operator. Although the ideal of finite jumps contains every $\mathbf 0^{(n)}$, it does not contain the uniform limit oracle
$\mathbf 0^{(\omega)}=\deg_T\!\left(\bigoplus_{n < \omega}0^{(n)}\right)$.
The uniformization operator adjoins this oracle only when all finite jump degrees are present. It is monotone but not Scott-continuous. Composing jump closure with one such gate yields closure ordinal $\omega\cdot 2$; gates at $\omega,2\omega,3\omega,\ldots$ yield closure ordinal $\omega^2$.

Thus non-uniform closure under relativized halting problems is Scott-continuous and reaches fixed ideals, while uniform coding of an entire prior hierarchy is infinitary, discontinuous, and reopens diagonalization. This gives a domain-theoretic semantics for the successor/limit distinction in transfinite Turing-jump hierarchies and links failures of Scott continuity with closure ordinals.
\end{abstract}

\noindent\textbf{Keywords.} Turing jump; Turing degrees; halting problem; domain theory; Scott continuity; Turing ideals; closure ordinals; hyperarithmetic hierarchy; ordinal analysis.

\medskip
\noindent\textbf{MSC 2020.} 03D28, 03D55, 06B35.

\section{Introduction}

The Turing jump is the canonical operation for producing a stronger degree of unsolvability from a given oracle. If $A\subseteq\omega$, its jump
\[
A'=\{e:\Phi_e^A(e)\downarrow\}
\]
is the halting problem relative to $A$. Its degree depends only on the Turing degree of $A$, and the induced operation
\[
\mathbf a\longmapsto\mathbf a'
\]
is monotone and strictly increasing:
\[
\mathbf a\ltT\mathbf a'.
\]
Consequently, the jump has no fixed point on the Turing degrees. This strictness is the degree-theoretic form of the recurrence of the halting diagonal: once an oracle solves one level of halting problems, machines using that oracle generate a new halting problem one jump higher.

Classical computability theory responds by iteration. The finite sequence
\[
\mathbf0,\mathbf0',\mathbf0'',\ldots
\]
organizes the finite arithmetical hierarchy. Effective transfinite iteration along recursive well-orders produces the hyperarithmetic hierarchy. At a recursive limit ordinal $\lambda$, one does not take another ordinary successor jump; instead one forms an effective join coding the preceding stages into a single oracle. The basic example is
\[
0^{(\omega)}=\bigjoin_{n<\omega}0^{(n)}.
\]
Turing's ordinal logics initiated this use of oracle computation in transfinite progressions, while the later work of Kleene and Spector supplied the machinery by which iterated jumps are organized through recursive ordinal notations \cite{Turing1939,Spector1955,Slaman2005}.

The purpose of this paper is to reinterpret the successor/limit structure as a fixed-point phenomenon in domain theory. The central observation is that the absence of fixed points for the jump depends on the \emph{type of object} to which the jump is applied. An individual degree can never contain its own jump. A collection of degrees, by contrast, can be closed under taking jumps. Turing ideals closed under the jump are classical objects, and they occur naturally as the second-order parts of $\omega$-models of arithmetical comprehension \cite{Simpson2009}.

We exploit the ideal completion of the Turing degrees. Ideal completion is a standard construction in domain theory: directed lower sets are treated as completed informational states, while principal ideals embed the original poset as compact approximants \cite{AbramskyJung1994,Scott1972}. Every monotone self-map of a poset has a canonical Scott-continuous extension to its ideal completion. Applied to the Turing jump, this yields
\[
\Gamma(I)=\down\{\mathbf a':\mathbf a\in I\}.
\]
Although the original jump has no fixed degree, $\Gamma$ has a least fixed point. Starting from the computable degree, the iteration is
\[
\down\mathbf0
\subsetneq
\down\mathbf0'
\subsetneq
\down\mathbf0''
\subsetneq\cdots,
\]
and its first limit is
\[
\bigcup_{n<\omega}\down\mathbf0^{(n)}.
\]
By Post's theorem, this is precisely the ideal of arithmetical degrees. It is already closed under the jump. The closure ordinal is exactly $\omega$.

This supplies a precise sense in which a ``Turing-jump fixed point'' exists:
\[
\text{no degree satisfies }\mathbf a'=\mathbf a,
\]
but
\[
\text{a domain of degrees can satisfy }\Gamma(I)=I.
\]
The fixed point does not solve diagonalization by producing a universal oracle. It contains, non-uniformly, a solver for every finite level of the hierarchy.

That qualification is essential. The arithmetical ideal contains each $\mathbf0^{(n)}$, but it does not contain $\mathbf0^{(\omega)}$. The latter packages all finite jumps into one oracle. Once this package is admitted, its own relativized halting problem
\[
0^{(\omega+1)}=(0^{(\omega)})'
\]
reopens the jump hierarchy.

Our second aim is to isolate this operation order-theoretically. We call it \emph{limit uniformization}. Given an increasing sequence of degrees and a degree uniformly coding that sequence, the associated operator waits until an ideal contains every member of the sequence and only then adjoins the uniform code. The requirement that \emph{all} stages be present has no finite witness. Correspondingly, the operator is monotone but fails Scott continuity.

The distinction can therefore be summarized as
\[
\boxed{\text{jump closure is Scott-continuous, whereas limit uniformization is not.}}
\]
This gives a domain-theoretic interpretation of the successor/limit distinction in transfinite jump iteration. At a successor stage one forms a relativized halting problem. At a limit stage one uniformly packages a cofinal family of prior stages. The former extends continuously to ideals; the latter detects completion of an infinite directed chain and therefore fails to commute with its supremum.

Combining these operations gives transfinite closure ordinals. One uniformization gate at $\omega$ yields closure ordinal $\omega\cdot2$. Gates at
\[
\omega,2\omega,3\omega,\ldots
\]
yield closure ordinal $\omega^2$.

The individual ingredients used here are classical: ideal completion, jump ideals, uniform upper bounds, and transfinite jump hierarchies all have established literatures \cite{Hodes1982,Slaman2005}. The point of the present construction is the semantic decomposition: non-uniform jump closure reaches a Scott-continuous fixed point, whereas uniformization of an infinite hierarchy is a discontinuous escape operation that restarts the diagonal process.

\section{Turing degrees and ideal completion}

\subsection{Turing degrees}

Let $\DT$ denote the partially ordered set of Turing degrees. We write
$\mathbf a\leT\mathbf b$ when every representative of $\mathbf a$ is Turing reducible to a representative of $\mathbf b$. The least degree is denoted $\mathbf0$. The join is
\[
\mathbf a\vee\mathbf b=\deg_T(A\join B).
\]

The Turing jump is the map
\[
J:\DT\to\DT,
\qquad
J(\mathbf a)=\mathbf a'.
\]
It is monotone,
\[
\mathbf a\leT\mathbf b
\quad\Longrightarrow\quad
\mathbf a'\leT\mathbf b',
\]
and strictly progressive,
\begin{equation}
\mathbf a\ltT\mathbf a'.
\label{eq:jump-strict}
\end{equation}
The strictness is the relativized undecidability of the halting problem. See, for example, \cite{KleenePost1954,ShoreSlaman1999}.

We write
\[
\mathbf a^{(0)}=\mathbf a,
\qquad
\mathbf a^{(n+1)}=(\mathbf a^{(n)})'.
\]

\subsection{Turing ideals}

\begin{definition}
A \emph{Turing ideal} is a nonempty set $I\subseteq\DT$ such that
\begin{enumerate}
\item if $\mathbf b\leT\mathbf a\in I$, then $\mathbf b\in I$;
\item if $\mathbf a,\mathbf b\in I$, then $\mathbf a\vee\mathbf b\in I$.
\end{enumerate}
\end{definition}

Equivalently, since $\DT$ is an upper semilattice, Turing ideals are directed lower subsets of $\DT$. Let
\[
\Idl(\DT)
\]
denote their poset under inclusion. For a degree $\mathbf a$, write
\[
\down\mathbf a=\{\mathbf b:\mathbf b\leT\mathbf a\}.
\]
The map
\[
\eta:\DT\to\Idl(\DT),
\qquad
\eta(\mathbf a)=\down\mathbf a
\]
embeds degrees as principal ideals.

The ideal completion is a dcpo. Directed suprema are unions:
\[
\bigvee_{i\in D}I_i=\bigcup_{i\in D}I_i
\]
for every directed family $(I_i)_{i\in D}$. In the usual domain-theoretic interpretation, the principal ideals are the finitely generated approximants and arbitrary ideals are their directed completions \cite{AbramskyJung1994}.

\subsection{Canonical lifting to the ideal completion}

We first record the general order-theoretic fact on which the construction rests.

\begin{proposition}[Ideal-completion extension]
\label{prop:ideal-extension}
Let $P$ be a poset and let $f:P\to P$ be monotone. Define
\[
\widehat f:\Idl(P)\to\Idl(P)
\]
by
\begin{equation}
\widehat f(I)
=\down f[I]
=\{y\in P:\exists x\in I\ (y\leq f(x))\}.
\label{eq:ideal-extension}
\end{equation}
Then $\widehat f$ is well defined and Scott-continuous. Moreover,
\begin{equation}
\widehat f(\down x)=\down f(x)
\label{eq:principal-extension}
\end{equation}
for every $x\in P$.
\end{proposition}

\begin{proof}
The set $\widehat f(I)$ is lower by definition. Suppose
$y_0\leq f(x_0)$ and $y_1\leq f(x_1)$ with $x_0,x_1\in I$. Since $I$ is directed, there is $x\in I$ with $x_0,x_1\leq x$. Monotonicity gives $f(x_0),f(x_1)\leq f(x)$, so $f(x)$ is an upper bound in $\widehat f(I)$ for $y_0,y_1$. Thus $\widehat f(I)$ is an ideal.

Let $(I_i)_{i\in D}$ be directed. Then
\begin{align*}
\widehat f\!\left(\bigcup_iI_i\right)
&=\{y:\exists x\in\bigcup_iI_i\ (y\leq f(x))\}\\
&=\bigcup_i\{y:\exists x\in I_i\ (y\leq f(x))\}\\
&=\bigcup_i\widehat f(I_i).
\end{align*}
Hence $\widehat f$ preserves directed suprema.

Finally, if $y\in\widehat f(\down x)$, then $y\leq f(z)$ for some $z\leq x$, and hence $y\leq f(x)$. Conversely, $f(x)$ is among the generators, which proves \eqref{eq:principal-extension}.
\end{proof}

The proposition exhibits the type shift that drives the paper: a progressive operation may fail to possess a fixed point among principal elements while its Scott-continuous extension possesses a nonprincipal fixed point in the completion.

\section{The jump-closure fixed point}

Apply \cref{prop:ideal-extension} to the Turing jump.

\begin{definition}
Define
\[
\Gamma:\Idl(\DT)\longrightarrow\Idl(\DT)
\]
by
\begin{equation}
\Gamma(I)=\down\{\mathbf a':\mathbf a\in I\}.
\label{eq:Gamma}
\end{equation}
\end{definition}

By \cref{prop:ideal-extension}, $\Gamma$ is Scott-continuous and
\begin{equation}
\Gamma(\down\mathbf a)=\down\mathbf a'.
\label{eq:Gamma-principal}
\end{equation}
Since $\mathbf a\leT\mathbf a'$ for every degree $\mathbf a$, we also have
\begin{equation}
I\subseteq\Gamma(I).
\label{eq:Gamma-inflationary}
\end{equation}
Thus $\Gamma$ is inflationary.

\begin{proposition}
\label{prop:jump-ideal}
For a Turing ideal $I$, the following are equivalent:
\begin{enumerate}
\item $\Gamma(I)=I$;
\item $I$ is closed under the Turing jump.
\end{enumerate}
\end{proposition}

\begin{proof}
If $\Gamma(I)=I$ and $\mathbf a\in I$, then
$\mathbf a'\in\Gamma(I)=I$. Conversely, if $I$ is jump-closed, then each generator $\mathbf a'$ in \eqref{eq:Gamma} already belongs to $I$, so downward closure gives $\Gamma(I)\subseteq I$. The reverse inclusion follows from \eqref{eq:Gamma-inflationary}.
\end{proof}

The fixed points of the lifted jump are therefore exactly the jump-closed Turing ideals.

For a degree $\mathbf a$, define
\begin{equation}
\operatorname{Arith}(\mathbf a)
=
\bigcup_{n<\omega}\down\mathbf a^{(n)}.
\label{eq:arith-a}
\end{equation}

\begin{theorem}[Least jump fixed point]
\label{thm:least-jump-fixed}
Starting from $\down\mathbf a$, the finite iterates of $\Gamma$ satisfy
\begin{equation}
\Gamma^n(\down\mathbf a)=\down\mathbf a^{(n)}
\label{eq:finite-iterates}
\end{equation}
for every $n<\omega$, and
\begin{equation}
\lfp_{\geq\down\mathbf a}(\Gamma)
=
\bigcup_{n<\omega}\down\mathbf a^{(n)}
=
\operatorname{Arith}(\mathbf a).
\label{eq:relative-lfp}
\end{equation}
The closure ordinal of this iteration is exactly $\omega$.
\end{theorem}

\begin{proof}
Equation \eqref{eq:finite-iterates} follows by induction from \eqref{eq:Gamma-principal}. Since $\Gamma$ is Scott-continuous,
\begin{align*}
\Gamma\!\left(\bigcup_{n<\omega}\down\mathbf a^{(n)}\right)
&=
\bigcup_{n<\omega}\Gamma(\down\mathbf a^{(n)})\\
&=
\bigcup_{n<\omega}\down\mathbf a^{(n+1)}\\
&=
\bigcup_{n<\omega}\down\mathbf a^{(n)}.
\end{align*}
Thus the union is fixed. It is least among fixed ideals containing $\down\mathbf a$, because every such ideal must contain every finite jump $\mathbf a^{(n)}$ and hence everything below one of them.

No finite stage is fixed, since \eqref{eq:jump-strict} gives
$\mathbf a^{(n)}\ltT\mathbf a^{(n+1)}$. Hence the first fixed stage is $\omega$.
\end{proof}

For $\mathbf a=\mathbf0$, Post's theorem identifies \eqref{eq:arith-a} with the degrees of the arithmetical sets \cite{Post1944}.

\begin{corollary}
\label{cor:arith-lfp}
The global least fixed point of $\Gamma$ is
\[
\lfp(\Gamma)
=
\Arith
=
\bigcup_{n<\omega}\down\mathbf0^{(n)},
\]
where $\Arith$ denotes the Turing ideal of arithmetical degrees.
\end{corollary}

This has the standard reverse-mathematical interpretation that the second-order part of an $\omega$-model of $\mathsf{ACA}_0$ is a Turing ideal closed under jump, while the arithmetical sets form the least such $\omega$-model \cite{Simpson2009}.

\begin{corollary}[No principal fixed point]
\label{cor:no-principal}
No principal Turing ideal is a fixed point of $\Gamma$.
\end{corollary}

\begin{proof}
If $\Gamma(\down\mathbf a)=\down\mathbf a$, then by \eqref{eq:Gamma-principal},
$\down\mathbf a'=\down\mathbf a$, hence $\mathbf a'=\mathbf a$, contradicting \eqref{eq:jump-strict}.
\end{proof}

Thus completion has not manufactured a degree that computes its own halting problem. The fixed point is intrinsically nonprincipal. It is better understood as a \emph{domain of solvers closed under producing the next relativized halting problem}.

\section{Uniformization at a limit}

The fixed point $\Arith$ does not contain $\mathbf0^{(\omega)}$. This is the key to continuing the hierarchy.

Each degree $\mathbf0^{(n)}$ belongs to $\Arith$, but the single oracle
\begin{equation}
0^{(\omega)}=\bigjoin_{n<\omega}0^{(n)}
\label{eq:omega-jump}
\end{equation}
is not arithmetical. The issue is uniformity. Membership of every finite stage in an ideal says
\[
\forall n\;\mathbf0^{(n)}\in I,
\]
but it does not provide one element of $I$ from which the whole sequence can be uniformly recovered. Classical work on uniform upper bounds for Turing ideals studies a related and substantially more general notion \cite{Hodes1982}; here we use only the canonical uniform code supplied by a specified jump hierarchy.

\begin{definition}[Uniformization gate]
\label{def:uniformization}
Let
\[
\mathbf a_0\ltT\mathbf a_1\ltT\mathbf a_2\ltT\cdots
\]
be a strictly increasing sequence of Turing degrees. Put
\begin{equation}
A=\bigcup_{n<\omega}\down\mathbf a_n.
\label{eq:A-chain}
\end{equation}
Let $\mathbf b$ satisfy
\begin{equation}
\mathbf a_n\leT\mathbf b\quad\text{for every }n,
\label{eq:b-upper}
\end{equation}
and
\begin{equation}
\mathbf b\notin A.
\label{eq:b-not-A}
\end{equation}
Define
\begin{equation}
U_{A,\mathbf b}(I)
=
\begin{cases}
\operatorname{Idl}(I\cup\{\mathbf b\}),&A\subseteq I,\\[3pt]
I,&A\nsubseteq I.
\end{cases}
\label{eq:uniformization}
\end{equation}
\end{definition}

The operator waits for completion of the entire chain and then makes the designated uniform upper bound available.

\begin{theorem}[Uniformization is discontinuous]
\label{thm:uniformization-discontinuous}
The operator $U_{A,\mathbf b}$ is monotone and inflationary, but it is not Scott-continuous.
\end{theorem}

\begin{proof}
Inflationarity is immediate. For monotonicity, suppose $I\subseteq J$. If $A\nsubseteq I$, then either $A\nsubseteq J$, in which case
$U(I)=I\subseteq J=U(J)$, or $A\subseteq J$, in which case
$I\subseteq J\subseteq U(J)$. If $A\subseteq I$, then $A\subseteq J$, and ideal generation is monotone.

For discontinuity, take
\[
I_n=\down\mathbf a_n.
\]
Then $(I_n)_{n<\omega}$ is an increasing chain and
\[
\bigcup_nI_n=A.
\]
For every $n$, $A\nsubseteq I_n$ because $\mathbf a_{n+1}\not\leT\mathbf a_n$. Hence $U(I_n)=I_n$ and
\begin{equation}
\bigcup_nU(I_n)=A.
\label{eq:U-union}
\end{equation}
On the other hand,
\[
U(A)=\operatorname{Idl}(A\cup\{\mathbf b\}).
\]
By \eqref{eq:b-upper}, every member of $A$ lies below $\mathbf b$, so
\[
U(A)=\down\mathbf b.
\]
By \eqref{eq:b-not-A}, this strictly contains $A$. Therefore
\[
U\!\left(\bigcup_nI_n\right)
\neq
\bigcup_nU(I_n),
\]
so $U$ is not Scott-continuous.
\end{proof}

The theorem identifies the exact informational source of the discontinuity: the condition that the entire chain has been completed cannot be witnessed at any finite member of the directed system.

\subsection{The \texorpdfstring{$\omega$}{omega}-jump gate}

Take
\[
\mathbf a_n=\mathbf0^{(n)}
\quad\text{and}\quad
\mathbf b=\mathbf0^{(\omega)},
\]
where $0^{(\omega)}$ is given by \eqref{eq:omega-jump}. Let
\[
A_\omega
=
\bigcup_{n<\omega}\down\mathbf0^{(n)}
=
\Arith.
\]
Clearly $\mathbf0^{(n)}\leT\mathbf0^{(\omega)}$ for every $n$.

Moreover,
\begin{equation}
\mathbf0^{(\omega)}\notin A_\omega.
\label{eq:omega-not-arith}
\end{equation}
Indeed, if $\mathbf0^{(\omega)}\leT\mathbf0^{(n)}$ for some $n$, then since
$\mathbf0^{(n+1)}\leT\mathbf0^{(\omega)}$, we would have
$\mathbf0^{(n+1)}\leT\mathbf0^{(n)}$, contradicting strictness of the jump.

Let $U_\omega$ denote the corresponding instance of \eqref{eq:uniformization}. Then
\begin{equation}
U_\omega(A_\omega)=\down\mathbf0^{(\omega)}.
\label{eq:omega-gate-fires}
\end{equation}
Thus $A_\omega$, although fixed by jump closure, can be escaped by uniformizing the information distributed through it. Once this happens,
\[
\Gamma(\down\mathbf0^{(\omega)})
=
\down\mathbf0^{(\omega+1)},
\]
and diagonalization resumes.

Schematically,
\begin{equation}
\boxed{
\text{jump iteration}
\longrightarrow
\text{jump-closed fixed ideal}
\longrightarrow
\text{uniform limit code}
\longrightarrow
\text{new jump iteration}.}
\label{eq:escape-schema}
\end{equation}

The transfinite hierarchy is therefore not a sequence in which diagonalization is eventually defeated. Fixed closure and renewed diagonalization alternate.

\section{A closure ordinal of \texorpdfstring{$\omega\cdot2$}{omega times 2}}

We now combine the two operations into a single monotone transfinite induction. Define
\begin{equation}
\Theta_2=U_\omega\circ\Gamma.
\label{eq:theta2}
\end{equation}
Since both factors are monotone and inflationary, so is $\Theta_2$. By \cref{thm:uniformization-discontinuous}, it is not Scott-continuous.

Starting with
\[
I_0=\down\mathbf0,
\]
define
\begin{equation}
I_{\alpha+1}=\Theta_2(I_\alpha),
\label{eq:successor-iteration}
\end{equation}
and, at limit ordinals,
\begin{equation}
I_\lambda=\bigcup_{\alpha<\lambda}I_\alpha.
\label{eq:limit-iteration}
\end{equation}

For $n<\omega$, the gate has not fired, so
\begin{equation}
I_n=\down\mathbf0^{(n)}.
\label{eq:first-block}
\end{equation}
Consequently,
\begin{equation}
I_\omega=A_\omega=\Arith.
\label{eq:stage-omega}
\end{equation}
At this point ordinary jump closure has stabilized,
$\Gamma(I_\omega)=I_\omega$, but the uniformization gate fires:
\begin{equation}
I_{\omega+1}=\down\mathbf0^{(\omega)}.
\label{eq:omega-plus-one}
\end{equation}
Successive stages then satisfy
\begin{equation}
I_{\omega+n+1}
=
\down\mathbf0^{(\omega+n)}
\qquad(n<\omega).
\label{eq:second-block}
\end{equation}

Put
\begin{equation}
A_{\omega\cdot2}
=
\bigcup_{n<\omega}\down\mathbf0^{(\omega+n)}.
\label{eq:A-omega2}
\end{equation}
(The earlier finite-jump ideals are already contained in this union.)

\begin{theorem}
\label{thm:omega-times-two}
The closure ordinal of $\Theta_2$ from $\down\mathbf0$ is exactly $\omega\cdot2$, and
\[
\lfp(\Theta_2)=A_{\omega\cdot2}.
\]
\end{theorem}

\begin{proof}
Equations \eqref{eq:first-block}--\eqref{eq:second-block} determine the iteration below $\omega\cdot2$. At the limit,
\[
I_{\omega\cdot2}=A_{\omega\cdot2}.
\]
This ideal is jump-closed. If $\mathbf x\leT\mathbf0^{(\omega+n)}$, then
\[
\mathbf x'\leT\mathbf0^{(\omega+n+1)},
\]
which is again represented in the union. Hence
\[
\Gamma(A_{\omega\cdot2})=A_{\omega\cdot2}.
\]
The $U_\omega$ gate adds nothing new because $\mathbf0^{(\omega)}\in A_{\omega\cdot2}$. Therefore
\[
\Theta_2(A_{\omega\cdot2})=A_{\omega\cdot2}.
\]
No earlier stage is fixed: every principal stage is moved by strictness of the jump, while $A_\omega$ is moved by the uniformization gate. Thus the first fixed stage is $\omega\cdot2$.
\end{proof}

The ordinal $\omega\cdot2$ has a direct meaning in this semantics. The first $\omega$-block closes under all finite jumps. Limit uniformization packages that block into $0^{(\omega)}$. A second $\omega$-block then closes under every finite jump relative to this new oracle.

Notice that $A_{\omega\cdot2}$ again contains every $\mathbf0^{(\omega+n)}$ but does not contain the next uniform limit oracle $\mathbf0^{(\omega\cdot2)}$. The same mechanism can therefore be repeated.

\section{Finite blocks and closure ordinal \texorpdfstring{$\omega^2$}{omega squared}}

We now iterate the construction through all finite multiples of $\omega$. Since all ordinals below $\omega^2$ have canonical notation $\omega k+n$ with $k,n<\omega$, no issues of notation invariance arise at this stage.

For $\alpha<\omega^2$, write
\[
\mathbf j_\alpha=\mathbf0^{(\alpha)}.
\]
Successor stages satisfy
\begin{equation}
\mathbf j_{\alpha+1}=\mathbf j_\alpha'.
\label{eq:j-successor}
\end{equation}
For $k\geq1$, choose the standard limit representative
\begin{equation}
0^{(\omega k)}
=
\bigjoin_{n<\omega}0^{(\omega(k-1)+n)}.
\label{eq:block-limit-join}
\end{equation}
Because the displayed sequence is cofinal below $\omega k$ and higher finite jumps compute all lower stages, this degree uniformly codes the entire preceding initial segment.

For a nonzero limit $\lambda=\omega k<\omega^2$, put
\begin{equation}
A_\lambda
=
\bigcup_{\beta<\lambda}\down\mathbf j_\beta.
\label{eq:A-lambda}
\end{equation}
Define the corresponding gate
\begin{equation}
U_\lambda(I)
=
\begin{cases}
\operatorname{Idl}(I\cup\{\mathbf j_\lambda\}),&A_\lambda\subseteq I,\\[3pt]
I,&A_\lambda\nsubseteq I.
\end{cases}
\label{eq:U-lambda}
\end{equation}
Every $U_\lambda$ is monotone and non-Scott-continuous by \cref{thm:uniformization-discontinuous}; a cofinal chain witnessing discontinuity is
\[
\down\mathbf j_{\omega(k-1)}
\subseteq
\down\mathbf j_{\omega(k-1)+1}
\subseteq\cdots.
\]

For $m\geq1$, define a combined uniformization operator
\begin{equation}
\mathcal U_m(I)
=
\operatorname{Idl}\!\left(
I\cup
\left\{
\mathbf j_{\omega k}:
1\leq k<m
\ \text{and}\ 
A_{\omega k}\subseteq I
\right\}
\right).
\label{eq:Um}
\end{equation}
Thus $\mathcal U_m$ contains gates at
\[
\omega,2\omega,\ldots,\omega(m-1).
\]
Put
\begin{equation}
\Theta_m=\mathcal U_m\circ\Gamma.
\label{eq:Thetam}
\end{equation}
For $m=1$, there are no gates and $\Theta_1=\Gamma$.

\begin{theorem}[Finite block theorem]
\label{thm:finite-block}
For every integer $m\geq1$, the closure ordinal of $\Theta_m$, starting from $\down\mathbf0$, is
\[
\omega m,
\]
and its least fixed point is
\begin{equation}
A_{\omega m}
=
\bigcup_{\beta<\omega m}\down\mathbf0^{(\beta)}.
\label{eq:A-omegam}
\end{equation}
\end{theorem}

\begin{proof}
The first block is the ordinary jump iteration:
\[
I_n=\down\mathbf j_n
\qquad(n<\omega),
\]
so $I_\omega=A_\omega$. If $m=1$, there is no gate at $\omega$, and \cref{thm:least-jump-fixed} gives the result.

Assume $m>1$. At $A_\omega$, jump closure adds nothing while the gate $U_\omega$ adds $\mathbf j_\omega$. Hence
\[
I_{\omega+1}=\down\mathbf j_\omega.
\]
Repeated jump closure gives
\[
I_{\omega+n+1}=\down\mathbf j_{\omega+n}
\qquad(n<\omega).
\]
At the next limit,
\[
I_{\omega\cdot2}=A_{\omega\cdot2}.
\]
If $2<m$, the $U_{\omega\cdot2}$ gate fires and yields
\[
I_{\omega\cdot2+1}=\down\mathbf j_{\omega\cdot2}.
\]
Continuing inductively, for every $1\leq k<m$,
\begin{equation}
I_{\omega k}=A_{\omega k}
\label{eq:block-boundary}
\end{equation}
and
\begin{equation}
I_{\omega k+n+1}=\down\mathbf j_{\omega k+n}
\qquad(n<\omega).
\label{eq:inside-block}
\end{equation}
At stage $\omega m$ we obtain $A_{\omega m}$.

This ideal is jump-closed. If $\mathbf x\leT\mathbf j_\beta$ for some $\beta<\omega m$, then
$\mathbf x'\leT\mathbf j_{\beta+1}$, and $\beta+1<\omega m$ because $\omega m$ is a limit ordinal. Thus $\Gamma(A_{\omega m})=A_{\omega m}$.

Every gate appearing in $\mathcal U_m$ corresponds to some $\omega k<\omega m$, and its limit degree $\mathbf j_{\omega k}$ already belongs to $A_{\omega m}$. Hence $\mathcal U_m(A_{\omega m})=A_{\omega m}$, so $A_{\omega m}$ is fixed.

No earlier stage is fixed. At a successor position strict jump progression forces movement. The only nonzero limit ordinals below $\omega m$ are the block boundaries $\omega k$ for $1\leq k<m$, and at each of those the corresponding uniformization gate forces movement. Hence the closure ordinal is exactly $\omega m$.
\end{proof}

We can activate all finite block gates simultaneously. Define
\begin{equation}
\mathcal U_{<\omega^2}(I)
=
\operatorname{Idl}\!\left(
I\cup
\left\{
\mathbf j_{\omega k}:
k\geq1
\ \text{and}\ 
A_{\omega k}\subseteq I
\right\}
\right)
\label{eq:Uomega2}
\end{equation}
and
\begin{equation}
\Theta_{<\omega^2}
=
\mathcal U_{<\omega^2}\circ\Gamma.
\label{eq:Thetaomega2}
\end{equation}

\begin{theorem}[$\omega^2$-closure]
\label{thm:omega-square}
The closure ordinal of $\Theta_{<\omega^2}$ from $\down\mathbf0$ is exactly
\[
\omega^2,
\]
and its least fixed point is
\begin{equation}
A_{\omega^2}
=
\bigcup_{\beta<\omega^2}\down\mathbf0^{(\beta)}.
\label{eq:A-omega-square}
\end{equation}
\end{theorem}

\begin{proof}
For every finite $m$, the iteration agrees below $\omega m$ with the operator $\Theta_m$. Hence
\[
I_{\omega m}=A_{\omega m}.
\]
Taking the directed union over $m$ gives
\[
I_{\omega^2}
=
\bigcup_{m<\omega}A_{\omega m}
=
A_{\omega^2}.
\]
The ideal $A_{\omega^2}$ is jump-closed because $\beta<\omega^2$ implies $\beta+1<\omega^2$. Thus
\[
\Gamma(A_{\omega^2})=A_{\omega^2}.
\]
Every limit degree explicitly adjoined by $\mathcal U_{<\omega^2}$ has the form $\mathbf j_{\omega k}$ for finite $k$, hence already belongs to $A_{\omega^2}$. Therefore
\[
\mathcal U_{<\omega^2}(A_{\omega^2})=A_{\omega^2}.
\]
So $A_{\omega^2}$ is fixed.

If $\alpha<\omega^2$, then either $\alpha=\omega k+n$ with $n>0$, where successor jumping prevents stabilization, or $\alpha=\omega k$ with $k\geq1$, where the corresponding uniformization gate prevents stabilization. Therefore no earlier stage is fixed.
\end{proof}

For every $m\geq2$, and likewise for $\Theta_{<\omega^2}$, the operator is necessarily non-Scott-continuous. This follows abstractly because a Scott-continuous self-map of a pointed dcpo reaches its least fixed point at or before its ordinary $\omega$-chain. More concretely, the chain
\[
\down\mathbf0
\subseteq
\down\mathbf0'
\subseteq\cdots
\]
already witnesses the failure: the operator applied to the union can adjoin $\mathbf0^{(\omega)}$, whereas no finite stage can.

\section{Halting problems, transfinite jumps, and ordinal analysis}

\subsection{What is fixed?}

There are three distinct objects that should not be conflated.

First, an individual degree
\[
\mathbf0^{(n)},\quad
\mathbf0^{(\omega)},\quad\ldots
\]
is never fixed under the jump.

Second, an ideal such as
\[
A_\omega
=
\bigcup_{n<\omega}\down\mathbf0^{(n)}
\]
is fixed under non-uniform jump closure:
\[
\Gamma(A_\omega)=A_\omega.
\]

Third, the uniform limit degree $\mathbf0^{(\omega)}$ packages the preceding hierarchy into one oracle and therefore lies outside the fixed ideal it codes. The structure is
\[
A_\omega
\overset{U_\omega}{\longmapsto}
\down\mathbf0^{(\omega)}
\overset{\Gamma}{\longmapsto}
\down\mathbf0^{(\omega+1)}
\longmapsto\cdots.
\]

A fixed ideal means that every individual member's halting problem is represented somewhere else inside the same domain. It does \emph{not} mean that the domain has a single member solving all those problems uniformly. Once such a uniform member is introduced, it becomes subject to its own diagonal.

This gives the main conceptual principle:
\begin{center}
\fbox{\parbox{0.82\linewidth}{\centering Closure under relative halting can stabilize non-uniformly. Uniformizing the stabilized hierarchy creates a new oracle and therefore a new relativized halting problem.}}
\end{center}

\subsection{Successor and limit stages}

Classical transfinite jump iteration already distinguishes these operations. At a successor,
\begin{equation}
0^{(\alpha+1)}=(0^{(\alpha)})'.
\label{eq:transfinite-successor}
\end{equation}
At a recursive limit, an effective join codes a cofinal family of preceding stages,
\begin{equation}
0^{(\lambda)}\equivT
\bigjoin_{n<\omega}0^{(\lambda[n])},
\label{eq:transfinite-limit}
\end{equation}
where $\lambda[0]<\lambda[1]<\cdots$ is an effective fundamental sequence cofinal in $\lambda$. The exact formulation is carried out using recursive ordinal notations; Spector's work gives the relevant invariance results for the resulting degrees \cite{Spector1955,Slaman2005}.

Equations \eqref{eq:transfinite-successor} and \eqref{eq:transfinite-limit} are computationally different, and the present semantics turns the difference into a continuity distinction. Successor jump closure extends to a Scott-continuous operation on ideals. Limit uniformization asks whether an entire cofinal family has become available. It therefore depends on infinitary information and fails Scott continuity.

In the language of inductive definitions, a limit gate has an infinitary premise. Aczel's analysis of monotone induction emphasizes precisely the relation between rule structure, transfinite stages, proof-tree ranks, and closure ordinals \cite{Aczel1977}. The present construction suggests the schematic correspondence
\[
\boxed{\text{successor jump}\ \leftrightarrow\ \text{continuous closure}},
\qquad
\boxed{\text{limit jump}\ \leftrightarrow\ \text{discontinuous uniformization}}.
\]

\subsection{Why \texorpdfstring{$\omega^2$}{omega squared} is not a degree invariant}

\Cref{thm:omega-square} does not assert that $\mathbf0^{(\omega)}$ has intrinsic computational complexity $\omega$, that $\mathbf0^{(\omega\cdot2)}$ has intrinsic complexity $\omega\cdot2$, or that a Turing degree determines a unique domain-theoretic closure ordinal.

Closure ordinals are invariants of operators and approximation schemes, not of Turing degrees alone. The ordinal $\omega^2$ appears because the operator has been given a particular dependency architecture:
\begin{enumerate}
\item close under successor jumps through an $\omega$-block;
\item wait for the whole block;
\item uniformly code it;
\item begin a new block;
\item repeat this process $\omega$ many times.
\end{enumerate}
Different operational semantics can attach different closure ordinals to the same eventual degree-theoretic information. This is exactly where the connection to ordinal analysis becomes substantive: ordinals measure the architecture of the inductive process by which information becomes available.

The appropriate question is therefore not ``What ordinal is the degree $\mathbf a$?'' but
\begin{center}
\fbox{\parbox{0.84\linewidth}{\centering What closure ordinal is required by a specified rule system for generating or uniformly organizing degrees up to $\mathbf a$?}}
\end{center}
This is structurally analogous to the use of closure ordinals in theories of inductive definitions.

\subsection{Toward the Church--Kleene boundary}

The hierarchy above stops at $\omega^2$ only to keep the mechanism transparent. Classical hyperarithmetic theory iterates the jump through every recursive ordinal below the Church--Kleene ordinal $\omega_1^{CK}$ \cite{Spector1955,Slaman2005}. At each recursive limit $\lambda$, an effective fundamental sequence gives an increasing chain
\[
\down\mathbf0^{(\lambda[0])}
\subseteq
\down\mathbf0^{(\lambda[1])}
\subseteq\cdots.
\]
Its union need not contain the uniform limit degree $\mathbf0^{(\lambda)}$. The associated limit gate is therefore non-Scott-continuous by the same argument as \cref{thm:uniformization-discontinuous}.

This suggests a general program. Fix an effective ordinal notation system and equip recursive limit notations with uniformization rules. The resulting monotone operator on an appropriate domain of degree information will have a closure ordinal determined by the dependency structure of those limit rules.

Developing this systematically through arbitrary recursive ordinals requires care. Stage indices need not coincide naively with the ordinal labels on jump degrees: application of a discontinuous limit rule occurs only after the directed supremum at its triggering stage. Moreover, notation invariance must be derived from the classical machinery of hyperarithmetic hierarchies rather than assumed.

\begin{problem}
\label{prob:recursive-alpha}
Construct a canonical family of monotone domain operators $\Theta_\alpha$ for recursive ordinals $\alpha$ such that:
\begin{enumerate}
\item the fixed point of $\Theta_\alpha$ contains exactly an intended initial segment of the transfinite jump hierarchy;
\item its closure rank is invariant under the chosen representation of $\alpha$, up to an explicitly characterized ordinal transformation;
\item failures of Scott continuity occur precisely at effective limit-uniformization nodes.
\end{enumerate}
\end{problem}

A satisfactory solution would connect closure ordinals of domain operators directly to the structure of the hyperarithmetic jump hierarchy. At the Church--Kleene boundary, a further transition occurs: there is no recursive notation for $\omega_1^{CK}$ itself. Kleene's $\mathcal O$ and the hyperjump then encode a higher level of effective ordinal information. Whether this transition admits an analogous fixed-point/escape semantics is a natural continuation of the present framework.

\subsection{Closure ordinals versus proof-theoretic ordinals}

The closure ordinal of an inductive operator should not be identified with the proof-theoretic ordinal of a formal theory describing that operator. A closure ordinal measures how long a specified semantic induction takes to stabilize. A proof-theoretic ordinal measures the transfinite induction strength of a formal system. Classical theories of inductive definitions sharply distinguish these two roles \cite{Aczel1977}.

The ordinal-analysis program suggested here therefore has two layers. First, classify operator schemes by semantic closure ordinals
\[
\cl(\Theta).
\]
Second, ask for the proof-theoretic strength of theories capable of proving the existence, invariance, or well-foundedness of the corresponding fixed-point constructions. The present paper addresses the first layer only, and explicitly only through $\omega^2$.

\section{Relation to bounded informational fixed points}

A simpler domain-theoretic halting construction begins with finite observations of a single computation and advances them one step at a time. If $p_n$ records the behavior of a machine through the first $n$ stages, a local extension operator can produce an ascending chain
\[
p_0\sqsubseteq p_1\sqsubseteq p_2\sqsubseteq\cdots
\]
whose Scott supremum is reached at $\omega$. The important feature of such an operator is not a purported lower bound on literal self-simulation time, but its \emph{finite dependency structure}: each newly defined coordinate depends on finite information from the previous approximation.

The jump-ideal construction is the degree-theoretic analogue of this mechanism. The operator $\Gamma$ is Scott-continuous, so its iterative closure is exhausted by an $\omega$-chain. The uniformization gates introduced above mark the first point at which the dependency structure ceases to be finite: a gate requires the entire cofinal chain before it can fire. Consequently, closure ordinal and continuity change together.

This comparison suggests a more general hierarchy of informational operators:
\[
\begin{aligned}
\text{finite witness dependence} &\Longrightarrow \text{ Scott continuity and closure by }\omega,\\[3pt]
\text{cofinal-limit dependence} &\Longrightarrow \text{ non-Scott-continuity and transfinite closure},\\[3pt]
\text{nested effective limit dependence} &\Longrightarrow \text{ higher recursive closure ranks}.
\end{aligned}
\]
The Turing jump supplies a canonical family of computationally meaningful successor operations with which to populate such a hierarchy.

\section{Discussion}

The usual statement that the Turing jump has no fixed point is correct but incomplete as a description of iterated halting phenomena. It applies to individual degrees. Passing to the ideal completion changes the fixed-point structure without changing the underlying impossibility theorem.

The canonical ideal extension
\[
\Gamma(I)=\down\{\mathbf a':\mathbf a\in I\}
\]
is Scott-continuous. Its least fixed point above $\mathbf a$ is
\[
\bigcup_{n<\omega}\down\mathbf a^{(n)}.
\]
For $\mathbf a=\mathbf0$, this is the ideal of arithmetical degrees. There is therefore a precise sense in which the finite jump hierarchy converges to a fixed point. It is not the false equation
\[
\mathbf0^{(\omega)}=(\mathbf0^{(\omega)})'.
\]
It is the true equation
\[
\Gamma(\Arith)=\Arith.
\]

The contrast with the uniform limit oracle is fundamental:
\[
\mathbf0^{(\omega)}\notin\Arith.
\]
The first equation says that the collection is closed under taking relativized halting problems one member at a time. The second says that the collection does not uniformly internalize its entire construction in a single oracle.

Limit uniformization converts the former kind of closure into the latter kind of object. Because doing so requires recognition of completion of an infinite directed chain, it is not Scott-continuous. The new oracle is then subject to the jump again.

The recurrence of the diagonal can thus be pictured as
\[
\text{successor jumps}
\longrightarrow
\text{fixed ideal}
\longrightarrow
\text{uniform limit oracle}
\longrightarrow
\text{new relative halting problem}.
\]
Diagonalization is not eliminated at the limit but relocated. Non-uniform closure can absorb every previous jump operation; uniformly coding that closure creates a new point from which a further diagonal can be taken.

Theorems \ref{thm:least-jump-fixed}, \ref{thm:omega-times-two}, and \ref{thm:omega-square} exhibit the first closure ordinals of this semantics:
\[
\omega,
\qquad
\omega\cdot2,
\qquad
\omega^2.
\]
The broader question is whether effective transfinite jump hierarchies admit a canonical domain-theoretic presentation in which recursive ordinal structure is recovered as the rank of precisely characterized failures of finite information and Scott continuity.

\section*{Acknowledgments}
Generative AI was used in writing this document.

\end{document}